\documentclass{amsart}

\usepackage{amsfonts,amstext,amssymb, mathtools, comment,amsmath,amsthm, graphicx}

\newcommand{\levy}{L\'{e}vy }
\newcommand{\p}{{\mathbb P}}
\newcommand{\e}{{\mathbb E}}
\newcommand{\D}{{\mathrm d}}
\newcommand{\ep}{{\epsilon}}

\newcommand{\ind}[1]{\mbox{\rm 1}_{\{#1\}}}

\newcommand{\matI}{\mathbb{I}}
\newcommand{\matO}{\mathbb{O}}
\newcommand{\diag}{{\rm diag}}
\newcommand{\1}{{\bs 1}}
\renewcommand{\t}{{\theta}}
\newcommand{\phib}{\boldsymbol \phi}
\newcommand{\bs}{\boldsymbol}
\newcommand{\bzero}{{\bs 0}}
\newcommand{\bv}{{\bs v}}
\newcommand{\bh}{{\bs h}}
\newcommand{\bpi}{{\bs \pi}}
\newcommand{\om}{{\omega}}

\newtheorem{theorem}{Theorem}[section]
\newtheorem{proposition}{Proposition}[section]
\newtheorem{lemma}{Lemma}[section]
\newtheorem{corollary}{Corollary}[section]
\newtheorem{remark}{Remark}[section]

\begin{document}
\title{A risk model with an observer in a Markov environment}

\author{Hansj{\"o}rg Albrecher}
\address{Department of Actuarial Science, University of Lausanne, CH-1015 Lausanne, Switzerland\\
Swiss Finance Institute, University of Lausanne, CH-1015 Lausanne, Switzerland}
\email{hansjoerg.albrecher@unil.ch}

\author{Jevgenijs Ivanovs}
\address{Department of Actuarial Science, University of Lausanne, CH-1015 Lausanne, Switzerland}
\email{jevgenijs.ivanovs@unil.ch}

\begin{abstract}We consider a spectrally-negative Markov additive process as a model of a risk process in random environment.
Following recent interest in alternative ruin concepts, we assume that ruin occurs 
when an independent Poissonian observer sees the process negative, where the observation rate may depend on the state of the environment.
Using an approximation argument and spectral theory we establish an explicit formula for the resulting survival probabilities in this general setting. 
 We also discuss an efficient evaluation of the involved quantities and provide a numerical illustration.
\end{abstract}



\maketitle

\section{Introduction}

In classical risk theory, ruin of an insurance portfolio is defined as the event that the surplus process becomes negative. 
In practice, it may be more reasonable to assume that the surplus value is not checked continuously, but at certain times only. If these times are not fixed deterministically, but are assumed to be epochs of a certain independent
renewal process, then one often still has sufficient analytical structure to obtain explicit expressions for ruin probabilities and related quantities, 
see~\cite{albrecher_thonhauser,actS} for corresponding studies in the framework of the Cram\'er-Lundberg risk model and Erlang inter-observation times.
An alternative ruin concept is studied in~\cite{albrecher_laundscham}, where negative surplus does not necessarily lead to bankruptcy, but bankruptcy is declared at the first instance of an inhomogeneous Poisson process with a rate depending on the surplus value, whenever it is negative.
When this rate is constant, this bankruptcy concept corresponds to the one in \cite{albrecher_thonhauser,actS} for exponential inter-observation times. 
Yet another related concept is the one of Parisian ruin, where ruin is only reported if the surplus process stays negative for a certain amount of time (see e.g. \cite{dawu,loeff}). If this time is assumed to be an 
independent exponential random variable instead of a deterministic value, one recovers the former models with exponential inter-observation times and constant bankruptcy rate function, respectively. Recently, simple 
expressions for the corresponding ruin probability have been derived when the surplus process follows a spectrally-negative \levy process, see \cite{landriault_occupation}.

In this paper we extend the above model and allow the surplus process to be a spectrally-negative Markov additive process.
The dynamics of such a process change according to an external environment process, modeled by a Markov chain, and changes of the latter may also cause a jump in the surplus process. 
We assume that ruin occurs when an independent Poissonian observer sees the surplus process negative, 
and we also allow the rate of observations to depend on the current state of the environment 
(one possible interpretation being that if the environment states refer to different economic conditions, a regulator may increase the observation rates in states of distress). 
Using an approximation argument and the spectral theory for Markov additive processes, we explicitly calculate for any initial capital the survival probability and the probability to reach a given level before ruin in this model. 
The resulting formulas turn out to be quite simple. At the same time, these formulas provide information on certain occupation times of the process, which may be of independent theoretical interest.

In Section~\ref{sec:model} we introduce the model and the considered quantities in more detail. 
Section~\ref{sec:review} gives a brief summary of general fluctuation results for Markov additive processes that are needed later on. 
In Section \ref{secres} we state our main results and discuss their relation with previous results, and the proofs are given in Section \ref{secpro}. 
In Section \ref{seccla} we reconsider the classical ruin concept and show how the present results implicitly extend the classical simple formula for the ruin probability 
with zero initial capital to the case of a Markov additive surplus process. 
Finally, in Section \ref{secnum} we give a numerical illustration of the results for our relaxed ruin concept in a Markov-modulated Cram\'er-Lundberg model.  

\section{The model}\label{sec:model}

Let $(X(t),J(t)),t\geq 0$ be a Markov additive process (MAP), where $X(t)$ is a surplus process and $J(t)$ is an irreducible Markov chain on $n$ states representing the environment, see e.g.~\cite{asm_ruin}.
While $J(t)=i$, $X(t)$ evolves as some \levy process $X_i(t)$, and $X(t)$ has a jump distributed as $U_{ij}$ when $J(t)$ switches from $i$ to $j$. Consequently, $X(t)$ has stationary and independent increments given the corresponding states of the environment. We assume that $X(t)$ has no positive jumps, and that none of the processes $X_i(t)$ is a non-increasing \levy process. 
The latter assumption allows to simplify notation and to avoid some tedious algebraic manipulations. 
Note that the Markov-modulated Cram\'er-Lundberg risk model with \begin{align}\label{cl}
&X(t) = u+\int_0^t c_{J(v)}\, dv-  \sum_{j=1}^{N(t)}Y_{j}
\end{align} is a particular case of the present framework, where $u$ is the initial capital of an insurance portfolio, $c_i>0$ is the premium density in state $i$, 
$N(t)$ is an inhomogeneous Poisson process with claim arrival intensity $\beta_i$ in state $i$, and $Y_{j}$ are independent claim sizes 
with distribution function $F_i$ if at the time of occurrence the environment is in state $i$ (in this case $U_{ij}\equiv 0$ for all $i,j$), see \cite{asm_ruin}.

Write $\e_u[Y;J(t)]$ for a matrix with $ij$th element $\e(Y\ind{J(t)=j}|J(0)=i,X(0)=u)$, where $Y$ is an arbitrary random variable, 
and $\p_u[A,J(t)]=\e_u[{\rm 1}_A;J(t)]$ for the probability matrix corresponding to an event~$A$. If $u=0$, then we simply drop the subscript. 
We write $\matI,\matO,\1,\bzero$ for an identity matrix, a zero matrix, a column vector of ones and a column vector of zeros of dimension $n$, respectively.
For $x\geq 0$ define the first passage time above $x$ (below $-x$) by 
\[\tau_x^\pm=\inf\{t\geq 0:\pm X(t)>x\}.\]

As in \cite{actS} we assume that ruin occurs when an independent Poissonian observer sees $X(t)$ negative, 
where in our setup the rate of observations depends on the state of~$J(t)$, i.e.\ the rate is $\om_{J(t)}\geq 0$ for given $\om_1,\ldots,\om_n$.
Recall that a Poisson process of rate $\om$ has no jumps (observations) in some Borel set $B\subset [0,\infty)$ with probability $\exp(-\om \int_B\D t)$.
Hence the probability of survival (non-ruin) in our model with initial capital $u$ is given by the column vector
\begin{align}
&{\boldsymbol \phi}(u)=\e_u e^{-\sum_{j}\om_jA_j},\quad \text{where}\, A_j:=\int_0^\infty\ind{X(t)<0,J(t)=j}\D t,
\end{align}
which follows by conditioning on the $A_j$s. The $i$th component of this vector refers to the probability of survival with initial state $J(0)=i$.
Define for any $u\leq x$ the $n\times n$ matrix 
\begin{align}\label{eq:R}
&R(u,x):=\e_u [e^{-\sum_{j}\om_jA_j(x)};J(\tau_x^+)], \quad \text{with}\,\;A_j(x):=\int_0^{\tau_x^+}\ind{X(s)<0,J(s)=j}\D s,
\end{align}
so $R(u,x)$ is the matrix of probabilities of reaching level $x$ without ruin, when starting at level $u$. 

It is known that $X(t)/t$ converges to a deterministic constant $\mu$ (the asymptotic drift of $X(t)$) a.s.\ as $t\rightarrow \infty$, independently of the initial state $J(0)$.
If $\mu< 0$, then $X(t)\rightarrow-\infty$ a.s., so $A_j\rightarrow\infty$ a.s. for all~$j$, and consequently ruin is certain (unless all $\om_j=0$).
If $\mu\geq 0$ then $\tau_x^+<\infty$ a.s.\ for all $x$, and so 
\[\boldsymbol \phi(u)=\lim_{x\rightarrow \infty}R(u,x)\1.\]
Finally, note that $R(u,x)$ can be interpreted as a joint transform of the occupation times $A_j(x)$. 
Moreover, with the definition $R(x):=R(0,x)$, the strong Markov property and the absence of positive jumps give
\begin{equation}\label{eq:Rxy}
 R(x)R(x,y)=R(y)
\end{equation}
for $0\leq x\leq y$ (see also \cite{gly}). Hence $R(x,y)$ can be expressed in terms of $R(x)$ and $R(y)$, given that these matrices are invertible. 
That is, it suffices to study the matrix-valued function $R(x)$.

\begin{remark}
The present framework can be extended to include positive jumps of phase type, cf.~\cite{asm_ruin}. 
One can convert a MAP with positive jumps of phase type into a spectrally-negative MAP using so-called fluid embedding, 
which amounts to expansion of the state space of $J(t)$, see e.g.~\cite[Sec.~2.7]{thesis}. Next, we set $\om_i=0$ for all the new auxiliary states $i$ and compute the corresponding
survival probability vector for the new model, which -- when restricted to the original states -- yields the survival probabilities of interest.
\end{remark}

\section{Review of exit theory for MAPs}\label{sec:review}

Let us quickly recall the recently established exit theory for spectrally-negative MAPs, which is an extension of the one for scalar \levy processes (see e.g.~\cite[Sec.~8]{kyprianou}).
A spectrally-negative MAP $(X(t),J(t))$ is characterized by a matrix-valued function $F(\t)$ via $\e[e^{\t X(t)};J(t)]=e^{F(\t)t}$ for $\t\geq 0$.
We let $\bpi$ be the stationary distribution of $J(t)$.
It is not hard to see that $J(\tau_x^+),x\geq 0$ is a Markov chain and thus $$\p(J(\tau_x^+)=j|J(0)=i)=(e^{\Lambda x})_{ij}$$ for a certain $n\times n$ transition rate matrix $\Lambda$,
which can be computed using an iterative procedure or a spectral method, see~\cite{breuer_lambda,dauria_lambda} and references therein. 
It is easy to see that $J(\tau_x^+),x\geq 0$ is non-defective (with a stationary distribution $\bpi_\Lambda$) if and only if $\mu\geq 0$.

The two-sided exit problem for MAPs without positive jumps was solved in~\cite{ivanovs_scale}, where it is shown that
\[\p_u[\tau_x^+<\tau_0^-,J(\tau_x^+)]=W(u)W(x)^{-1}\]
for $0\leq u\leq x$ and $x>0$,
where $W(x),x\geq 0$ is a continuous matrix-valued function (called scale function) characterized by the transform
\begin{equation}\label{eq:transform}
 \int_0^\infty e^{-\t x}W(x)\D x=F(\t)^{-1}
\end{equation}
for $\t$ sufficiently large. It is known that $W(x)$ is non-singular for $x>0$ and so is $F(\t)$ in the domain of interest. 
In addition, 
\begin{equation}\label{eq:WRep}
W(x)=e^{-\Lambda x}L(x), 
\end{equation}
where $L(x)$ is a positive matrix increasing (as $x\rightarrow\infty$) to $L$, a matrix of expected occupation times at zero
(note that in the case of the Markov modulated Cram\'er-Lundberg model \eqref{cl}, $c_jL_{ij}$ provides the expected number of times when the surplus is 0 in state $j$ given $J(0)=i$ and $X(0)=0$). 
If $\mu\neq 0$, then $L$ has finite entries and is invertible.
Finally,
\begin{equation}\label{eq:Z}
 \e_u[e^{\t X(\tau_0^-)};\tau_0^-<\tau_x^+,J(\tau_0^-)]=Z(\t,u)-W(u)W(x)^{-1}Z(\t,x),
\end{equation}
where \[Z(\t,x)=e^{\t x}\left(\matI-\int_0^x e^{-\t y}W(y)\D yF(\t)\right)\]
is analytic in $\t$ for fixed $x\geq 0$ in the domain $\Re(\t)>0$.

Importantly, all the above identities hold for defective (killed) MAPs as well, 
i.e.\ when the state space of $J(t)$ is complemented by an absorbing `cemetery' state; 
the original states of $J(t)$ then form a transient communicating class, and the (killing) rate from a state $i$ into the absorbing state is $\om_i\geq 0$. 
We refer to \cite{ivanovs_killing} for applications of the killing concept in risk theory. 

Note that killed MAPs preserve stationarity and independence of increments given the environment state.
Furthermore, we get probabilistic identities of the following type:
\begin{equation}\label{eq:hat_lambda}
e^{\hat\Lambda x}=\hat\p[J(\tau_x^+)]=\e[e^{-\sum_{j}\om_j\int_0^{\tau_x^+}1_{\{J_t=j\}}\D t};J(\tau_x^+)],
\end{equation}
where $\hat\p$ and $\hat\Lambda$ refer to the killed process, and we are still concerned with the original $n$ states only. 
The right hand side of~\eqref{eq:hat_lambda} is similar to the definition of the matrix $R(x)$ in~\eqref{eq:R}; it is also the joint transform of certain occupation times. 
However, $R(x)$ is more complicated, as there the killing is only applied when the surplus process is below zero, 
so with the setup of this paper one leaves the class of defective MAPs (the increments now depend on the current value of $X(t)$).
Let us recall the relation between $F(\t)$ and its killed analogue $\hat F(\t)$:
\begin{align}\label{eq:hat_F}
&\hat F(\t)=F(\t)-\Delta, \qquad \Delta=\diag(\om_1,\ldots,\om_n).
\end{align}

Letting $\Delta_\bpi$ be a diagonal matrix with the stationary distribution vector $\bpi$ of $J$ on the diagonal, 
we note that $\tilde F(\t)=\Delta_\bpi^{-1}F(\t)^T\Delta_\bpi$ corresponds to a time-reversed process, 
which is again a spectrally-negative MAP (with no non-increasing \levy processes as building blocks) with the same asymptotic drift~$\mu$, see~\cite{asm_ruin}.
Using the characterization~\eqref{eq:transform} one can see that the corresponding scale function is given by $\widetilde W(x)=\Delta_\bpi^{-1}W(x)^T\Delta_\bpi$.

\section{Results}\label{secres}

The following main result determines the matrix of probabilities of reaching a level $x$ without ruin:
\begin{theorem}\label{thm:main}
For $x\geq 0$ we have
\[R(x)=\e [e^{-\sum_{j}\om_jA_j(x)};J(\tau_x^+)]=e^{\hat\Lambda x}\left(\matI-\int_0^xW(y)\Delta e^{\hat\Lambda y}\D y\right)^{-1},\] 
where $\hat \Lambda$ corresponds to the killed process with killing rates $\om_i\geq 0$ identified by $\hat F(\t)$ in~\eqref{eq:hat_F}.
\end{theorem}

The vector of survival probabilities according to our relaxed ruin concept has the following simple form:
\begin{theorem}\label{thm:limit}
Assume that the asymptotic drift $\mu>0$, all obervation rates $\om_i$ are positive, and $\Lambda$ and $\hat \Lambda$ do not have a common eigenvalue. 
Then the vector of survival probabilities is given by
\[\phib (0)=\lim_{x\rightarrow\infty}R(x){\bf 1}=U^{-1}\1,\]
where $U$ is the unique solution of
\begin{equation}\label{eq:U}
 \Lambda U-U\hat\Lambda=L\Delta.
\end{equation}
\end{theorem}
Equation \eqref{eq:Rxy} then immediately gives 
\begin{corollary}\label{cor:results}
Under the conditions of Theorem \ref{thm:limit} we have for every $u\geq 0$
$$\phib(u)=R(u)^{-1}\phib(0)$$ and for every $0\leq u\leq x$
\[R(u,x)=\left(\matI-\int_0^uW(y)\Delta e^{\hat\Lambda y}\D y\right)e^{\hat\Lambda(x-u)}\left(\matI-\int_0^xW(y)\Delta e^{\hat\Lambda y}\D y\right)^{-1}.\] 
\end{corollary}

Equation~\eqref{eq:U} is known as the Sylvester equation in control theory. Under the conditions of Theorem~\ref{thm:limit} it has a unique solution~\cite{rutherford},
which has full rank, because $L\Delta$ has full rank~\cite[Thm.~2]{deSouza}. Moreover, the solution $U$ can be found by solving a system of linear equations with $n^2$ unknowns.
With regard to coefficient matrices, there are two methods to compute $\Lambda$ and $\hat\Lambda$, see Section~\ref{sec:review}. 
In principle, the matrix $L$ can be obtained from $W(x)$, cf. \eqref{eq:WRep}. This method, however, is ineffective and numerically unstable.
In the following we give a more direct way of evaluating $L$.
\begin{proposition}\label{prop:L}
 Let $\mu\neq 0$. Then for a left eigenpair $(\gamma,\bh)$ of $-\Lambda$, i.e.\ $-\bh\Lambda=\gamma\bh$, it holds that 
\[\bh L=\lim_{q\downarrow 0}q\bh F(\gamma+q)^{-1}.\]
More generally, if $\bh_1,\ldots,\bh_j$ is a left Jordan chain of $-\Lambda$ corresponding to an eigenvalue $\gamma$, i.e.\ $-\bh_1\Lambda=\gamma\bh_1$ and $-\bh_i\Lambda=\gamma\bh_i+\bh_{i-1}$ for $i=2,\ldots j$, then
\[\bh_j L=\lim_{q\downarrow 0}q\sum_{i=0}^{j-1}\frac{1}{i!}\bh_{j-i}[F(q+\gamma)^{-1}]^{(i)}.\]
\end{proposition}

\begin{remark}
 Consider the special case $n=1$, i.e. $X(t)$ is a spectrally-negative \levy process with Laplace exponent $F(\t)=\log \e e^{\t X(1)}$, with observation rate $\om$. 
Then $\hat\Lambda=-\Phi(\om)$, 
where $\Phi(\cdot)$ is the right-inverse of $F(\t)$, i.e.\ $F(\Phi(\om))=\om$. According to Theorem~\ref{thm:main} we have
\begin{equation}\label{forr}
 R(x)=e^{-\Phi(\om) x}/\left(1-\om\int_0^xe^{-\Phi(\om) y}W(y)\D y\right)=1/Z(\Phi(\om),x).
\end{equation}
Note that $1/Z(\t,x)$ is a certain transform corresponding to $X(t)$ reflected at zero at the time of passage over level $x$, see~\cite{ivanovs_scale}, which may lead one to an alternative direct probabilistic derivation of \eqref{forr}.
Finally, if $\mu=\e X(1)>0$ then $\Lambda=0$ and hence $L=1/F'(0)=1/\mu$ according to Proposition~\ref{prop:L}.
Accordingly, in this case Theorem~\ref{thm:limit} reduces to
\[\phi(0)=\e \exp\left(-\om\int_0^\infty\ind{X(t)<0}\D t\right)=\frac{\Phi(\om)}{\om}\mu,\]
which coincides with~\cite[Thm.\ 1]{landriault_occupation}.
\end{remark}


\section{Proofs}\label{secpro}

The proofs rely on a spectral representation of the matrix $\hat\Lambda$, which we quickly review in the following.
Let $\bv_1,\ldots,\bv_j$ be a Jordan chain of $-\hat\Lambda$ corresponding to an eigenvalue $\gamma$, i.e.\ $-\hat\Lambda\bv_1=\gamma\bv_1$ and $-\Lambda\bv_i=\gamma\bv_i+\bv_{i-1}$ for $i=2,\ldots j$.
From the classical theory of Jordan chains we know that 
\begin{equation}\label{eq:classical_jp}
 e^{-\hat\Lambda x}\bv_j=\sum_{i=0}^{j-1}\frac{x^i}{i!}e^{\gamma x}\bv_{j-i}
\end{equation}
for any $x\in\mathbb R$ and $j=1,\ldots,k$, and in particular $e^{-\hat\Lambda x}\bv_1=e^{\gamma x}\bv_1$.
Moreover, this Jordan chain turns out to be a generalized Jordan chain of an analytic matrix function $\hat F(\theta),\Re(\theta)>0$ corresponding to a generalized eigenvalue~$\gamma$, 
i.e.\ for any $j=1,\ldots,k$ it holds that
\begin{equation}\label{eq:Fjp}
 \sum_{i=0}^{j-1}\frac{1}{i!}\hat F^{(i)}(\gamma)\bv_{j-i}=\sum_{i=0}^{j-1}\frac{1}{i!}F^{(i)}(\gamma)\bv_{j-i}-\Delta\bv_j=\bs 0
\end{equation}
and in particular $F(\gamma)\bv_1=\Delta\bv_1$, see~\cite{dauria_lambda} for details.

\begin{proof}[Proof of Proposition~\ref{prop:L}]
 Observe that $\bh e^{-\Lambda x}=e^{\gamma x}\bh$ and so~\eqref{eq:transform} and~\eqref{eq:WRep} yield
\[\bh F(\t)^{-1}=\int_0^\infty e^{-\t x}e^{\gamma x}\bh L(x)\D x\]
for large enough $\t$. Since $L(x)$ is bounded from above by $L$, this equation can be analytically continued to $\Re(\t)>\Re(\gamma)$ with $F(\t)$ non-singular.
Hence for small enough $q>0$ we can write
\[q\bh F(q+\gamma)^{-1}=q\int_0^\infty e^{-qx}\bh L(x)\D x=\bh \e L(e_q),\]
where $e_q$ is an exponentially distributed r.v.\ with parameter $q$. Letting $q\downarrow 0$ completes the proof of the first part.

According to~\eqref{eq:classical_jp} we have $\bh_{j-i} e^{-\Lambda x}=\sum_{k=0}^{j-i-1}\frac{x^k}{k!}e^{\gamma x}\bh_{j-i-k}$.
Next, consider
\begin{align*}
&\bh_{j-i}[F(\theta)^{-1}]^{(i)}=\int_0^\infty (-x)^ie^{-\t x}\bh_{j-i}e^{-\Lambda x}L(x)\D x\\&=\sum_{k=i}^{j-1}\frac{(-1)^i}{(k-i)!}\bh_{j-k}\int_0^\infty x^k e^{-\t x+\gamma x}L(x)\D x, 
\end{align*}
where differentiation under the integral sign can be justified using standard arguments. 
Finally,
\begin{align*}
&\sum_{i=0}^{j-1}\frac{1}{i!}\bh_{j-i}[F(\theta)^{-1}]^{(i)}=\sum_{k=0}^{j-1}\sum_{i=0}^k\frac{(-1)^i}{i!(k-i)!}\bh_{j-k}\int_0^\infty x^k e^{-\t x+\gamma x}L(x)\D x\\
&=\bh_j\int_0^\infty e^{-\t x+\gamma x}L(x)\D x, 
\end{align*}
because the second sum is $(1-1)^k=0$ for $k\geq 1$.
The final step of the proof is the same as in the case of $j=1$.
\end{proof}

The proof of Theorem~\ref{thm:main} relies on an approximation idea, which has already appeared in various papers, see e.g.~\cite{breuer_dividend,breuer_exit_MMBM,landriault_occupation}.
We consider an approximation $R_\ep(x)$ of the matrix $R(x)$. When computing the occupation times we start the clock when $X(t)$ goes below $-\ep$ (rather than 0), but stop it when $X(t)$ reaches the level~$0$.
Mathematically, we write, using the strong Markov property,
\begin{align*}
&R_\ep(x)=\p[\tau_x^+<\tau_\ep^-,J(\tau_x^+)]\\&+\int_{-\infty}^{-\ep} \left(\p[\tau_\ep^-<\tau_x^+,X(\tau_\ep^-)\in\D y,J(\tau_\ep^-)] 
\e_y[e^{-\sum_{j}\om_j\int_0^{\tau_0^+}\ind{J_t=j}\D t};J(\tau_0^+)]\right)R_\ep(x).
\end{align*}
Using the exit theory for MAPs discussed in Section \ref{sec:review} we note that the first term on the right is $W(\ep)W(x+\ep)^{-1}$
and the second, according to~\eqref{eq:hat_lambda}, is
\begin{align*}
\int_{0}^\infty \left(\p_\ep[\tau_0^-<\tau_{x+\epsilon}^+,-X(\tau_0^-)\in\D y,J(\tau_0^-)]e^{\hat\Lambda (y+\ep)}\right)R_\ep(x).
\end{align*}

By the monotone convergence theorem the approximating occupation times converge to $A_j(x)$ as $\ep\downarrow 0$, and then the dominated convergence theorem
implies convergence of the transforms: $R_\ep(x)\rightarrow R(x)$ as $\ep\downarrow 0$ for any $x>0$.
Hence we have
\begin{align}\label{eq:proof_main}
&W(x)\lim_{\ep\downarrow 0}\left(W(\ep)^{-1}\left[\matI-\int_{0}^\infty \left(\p_\ep[\tau_0^-<\tau_{x+\epsilon}^+,-X(\tau_0^-)\in\D y,J(\tau_0^-)]e^{\hat\Lambda(y+\ep)}\right)\right]\right)\\
&\times R(x)=\matI,\nonumber
\end{align}
where we also used continuity of $W(x)$.
We will need the following auxiliary result for the analysis of the above limit.
\begin{lemma}\label{lem:lim0}
Let $f(y),y\geq 0$ be a Borel function bounded around 0. Then
 \[\lim_{\ep\downarrow 0}W(\ep)^{-1}\int_0^\ep f(y)W(y)\D y=\matO.\]
\end{lemma}
\begin{proof}
Consider a scale function $\tilde W(x)=\Delta_\bpi^{-1}W(x)^T\Delta_\bpi$ of the time-reversed process. 
It is enough to show that 
$\lim_{\ep\downarrow 0}\int_0^\ep f(y)\tilde W(y)\D y\tilde W(\ep)^{-1}=0,$
but 
\[\int_0^\ep f(y)\tilde W(y)\D y\tilde W(\ep)^{-1}=\int_0^\ep f(y)\tilde\p_y(\tau_\ep^+<\tau_0^-;J(\tau_\ep^+))\D y,\]
which clearly converges to the zero matrix.
\end{proof}

\begin{proof}[Proof of Theorem~\ref{thm:main}]
First we provide a proof under a simplifying assumption and then we deal with the general case.

\underline{Part I:} Assume that $-\hat\Lambda$ has $n$ linearly independent eigenvectors $\bv$: $-\hat\Lambda\bv=\gamma\bv$. 
Considering~\eqref{eq:proof_main} we observe that the integral multiplied by $\bv$ is given by
\begin{align*}
&\int_{0}^\infty \left(e^{-\gamma(y+\ep)}\p_\ep[\tau_0^-<\tau_{x+\epsilon}^+,-X(\tau_0^-)\in\D y,J(\tau_0^-)]\right)\bv=\\
&e^{-\gamma\ep}\e_\ep [e^{\gamma X(\tau_0^-)};\tau_0^-<\tau_{x+\epsilon}^+,J(\tau_0^-)]\bv =e^{-\gamma\ep}(Z(\gamma,\ep)-W(\ep)W(x+\epsilon)^{-1}Z(\gamma,x+\epsilon))\bv,
\end{align*}
according to~\eqref{eq:Z}. 
Hence the limit in~\eqref{eq:proof_main} multiplied by $\bv$ is given by
\[\lim_{\ep\downarrow 0} W(\ep)^{-1}\int_0^\ep e^{-\gamma y}W(y)\D yF(\gamma)\bv+W(x)^{-1}Z(\gamma,x)\bv=W(x)^{-1}Z(\gamma,x)\bv,\]
according to the form of $Z(\gamma,\ep)$ and Lemma~\ref{lem:lim0}.
Finally, from~\eqref{eq:Fjp} we have
\[Z(\gamma,x)\bv=e^{\gamma x}\bv-\int_0^x W(y)\Delta e^{\gamma (x-y)}\bv\D y=\left(e^{-\hat\Lambda x}-\int_0^x W(y)\Delta e^{\hat\Lambda (y-x)}\D y\right)\bv,\]
which under assumption that there are $n$ linearly independent eigenvectors shows that
\[
 \left(e^{-\hat\Lambda x}-\int_0^x W(y)\Delta e^{\hat\Lambda (y-x)}\D y\right)R(x)=\matI,
\]
completing the proof.

\underline{Part II:}
In general we consider a Jordan chain $\bv_1,\ldots,\bv_j$ of $-\hat\Lambda$ corresponding to an eigenvalue $\gamma$.
Using~\eqref{eq:classical_jp} we see that the integral in~\eqref{eq:proof_main} multiplied by $\bv_j$ is given by
\[\sum_{i=0}^{j-1}\frac{1}{i!}\e_\ep [(X(\tau_0^-)-\ep)^ie^{\gamma(X(\tau_0^-)-\ep)};\tau_0^-<\tau_{x+\epsilon}^+,J(\tau_0^-)]\bv_{j-i},\]
where all the terms can be obtained by considering~\eqref{eq:Z} for $\t=\gamma$, multiplying it by $e^{-\ep\gamma}$ and taking derivatives with respect to $\gamma$.
Again Lemma~\ref{lem:lim0} allows to show that various terms converge to 0, which results in
\begin{equation}\label{eq:ZZ}
 \sum_{i=0}^{j-1}\frac{1}{i!}Z^{(i)}(\gamma,x)\bv_{j-i}
\end{equation}
for the expression on the left of $R(x)$ in~\eqref{eq:proof_main} when multiplied by $\bv_j$.
The definition of $Z(\gamma,x)$ leads to
\begin{align*}
 &Z^{(i)}(\gamma,x)=x^{i}e^{\gamma x}\matI-\sum_{k=0}^{i}\frac{i!}{k!(i-k)!}\int_0^x (x-y)^k e^{\gamma(x-y)}W(y)\D yF^{(i-k)}(\gamma).
\end{align*}
Plugging this in~\eqref{eq:ZZ}, interchanging summation and using~\eqref{eq:Fjp}, we can rewrite~\eqref{eq:ZZ} in the following way:
\[\sum_{i=0}^{j-1}\frac{1}{i!}x^{i}e^{\gamma x}\bv_{j-i}-\sum_{k=0}^{j-1}\frac{1}{k!}\int_0^x (x-y)^k e^{\gamma(x-y)}W(y)\D y\Delta\bv_{j-k},\] 
which is just 
\[\left(e^{-\hat\Lambda x}-\int_0^x W(y)\Delta e^{\hat\Lambda (y-x)}\D y\right)\bv_j\]
according to~\eqref{eq:classical_jp}. The proof is complete since there are $n$ linearly independent vectors in the corresponding Jordan chains.
\end{proof}

\begin{proof}[Proof of Theorem~\ref{thm:limit}]
First, we provide a proof under the assumption that both $-\Lambda$ and $-\hat\Lambda$ have semi-simple eigenvalues, and that the real parts of the eigenvalues of $-\hat\Lambda$ 
are large enough.
Assume for a moment that every eigenvalue $\gamma$ of $-\hat\Lambda$ is such that the transform~\eqref{eq:transform} holds for $\theta=\gamma$.
In the following we will study the limit of $M(x)=e^{\Lambda x}R(x)^{-1}$.

Consider an eigenpair $(\gamma,\bv)$ of $-\hat\Lambda$ and a left eigenpair $(\gamma^*,\bh^*)$ of $-\Lambda$, i.e.\
$-\hat\Lambda\bv=\gamma\bv$ and $-\bh^*\Lambda=\gamma^*\bh^*$.
Then Theorem~\ref{thm:main} implies
\[\bh^*M(x)\bv=\bh^*\left(\matI-\int_0^xe^{-\gamma y}W(y)\D y\Delta\right)\bv e^{(\gamma-\gamma^*) x},\]
where $\Re(\gamma)>\Re(\gamma^*)$ by the above assumption.
Note that the expression in brackets converges to a zero matrix, because of~\eqref{eq:transform} and~\eqref{eq:Fjp}. 
So we can apply L'H\^opital's rule to get
\[\lim_{x\rightarrow \infty}\bh^*M(x)\bv=\frac{1}{\gamma-\gamma^*}\lim_{x\rightarrow \infty}e^{-\gamma^*x}\bh^*W(x)\Delta\bv=\frac{1}{\gamma-\gamma^*}\bh^*L\Delta\bv,\]
where the second equality follows from~\eqref{eq:WRep}. 
Under assumption that all the eigenvalues of $\Lambda$ and $\hat\Lambda$ are semi-simple (there are $n$ eigenvectors in each case), this implies
that $M(x)$ converges to a finite limit $U$ and 
\[\Lambda U-U\hat\Lambda=L\Delta.\]
Since $M(x)^{-1}\1=R(x)\1$ is bounded and $U$ is invertible, we see that the former converges to $U^{-1}\1$.

\underline{Jordan chains:}
When some eigenvalues are not semi-simple, the proof follows the same idea, but the calculus becomes rather tedious. So we only present the main steps.
Consider an arbitrary Jordan chain $\bv_1,\ldots,\bv_k$ of $-\hat\Lambda$ with eigenvalue $\gamma$, and an arbitrary left Jordan chain $\bh^*_1,\ldots,\bh^*_m$ of $-\Lambda$ with eigenvalue~$\gamma^*$.
We need to show that $M(x)$ has a finite limit $U$ as $x\rightarrow\infty$, and that this $U$ satisfies
\[\bh^*_m(\Lambda U-U\hat\Lambda)\bv_k=(\gamma-\gamma^*)\bh^*_m U\bv_k-\bh^*_{m-1}U\bv_k+\bh^*_mU\bv_{k-1}=\bh^*_m L\Delta\bv_k,\]
where $\bh^*_0=\bv_0=\bs 0$ by convention.
For this we compute $\bh^*_i M(x) \bv_j$ using~\eqref{eq:classical_jp} and its analogue for the left chain, and take the limit using  L'H\^opital's rule, which is applicable 
because of~\eqref{eq:Fjp}. This then confirms that 
\[(\gamma-\gamma^*)\bh^*_m M(x)\bv_k-\bh^*_{m-1}M(x)\bv_k+\bh^*_mM(x)\bv_{k-1}\rightarrow \bh^*_m L\Delta\bv_k\]
and the result follows.

\underline{Analytic continuation:}
Finally, it remains to remove the assumption that the real part of every eigenvalue of $-\hat\Lambda$ is large enough. For some $q>0$ we can define new killing rates by
$\om_i(q)=\om_i+q$ and consider the corresponding new matrices $\hat\Lambda(q),\Delta(q)$ (note that $\Lambda$ and $L$ stay unchanged). By choosing $q$ large enough we can
ensure that the real parts of the zeros of $\det(F(\t)-\Delta(q))$ (in the right half complex plane) are arbitrarily large. These zeros are exactly the eigenvalues of $-\hat\Lambda(q)$,
and so the result of our Theorem holds for large enough $q$. 

We now use analytic continuation in $q$ in the domain $\Re(q)>-\min\{\om_1,\ldots,\om_n\}$.
In this domain $e^{\hat\Lambda(q)x}$ is analytic for every $x$, which follows from its probabilistic interpretation.
This and invertibility of $\hat\Lambda(q)$ can be used to show that $\hat\Lambda(q)$ is also analytic. 
Furthermore, one can show that only for a finite number of different $q$'s the matrices $\Lambda$ and $\hat\Lambda(q)$ can have common eigenvalues.
Now we express $U(q)=G(q)^{-1}L\Delta(q)$, where $G(q)$ is formed from the elements of $\Lambda$ and $\hat\Lambda(q)$, see e.g.~\cite{lancaster}. Hence $U(q)$
can be analytically continued to the domain of interest excluding the above finite set of points. Hence also $\bs \phib^q(0)=U(q)^{-1}\1$ in the latter domain,
where $U(q)$ is the unique solution of the corresponding Sylvester equation. In particular, this holds for $q=0$, and the proof is complete.
\end{proof}

\section{Remarks on classical ruin}\label{seccla}

Let us briefly return to the classical ruin concept, i.e. all $\om_i\rightarrow\infty$. From~\eqref{eq:Z}, the matrix of probabilities to reach level $x$ before ruin is in this case given by
\[\p_u[\tau_x^+<\tau_0^-,J(\tau_x^+)]=\matI-Z(0,u)+W(u)W(x)^{-1}Z(0,x),\]
which for $u=0$ reduces to $W(0)W(x)^{-1}Z(0,x)$.
It is known that $W(0)$ is a diagonal matrix with $W_{ii}(0)$ equal to $0$ or $1/c_i$ according to $X_i$ having unbounded variation or bounded variation on compacts, and $c_i>0$ being the linear drift of $X_i$ (the premium density in case of \eqref{cl}).

In order to obtain survival probabilities when $\mu>0$ we need to compute 
\[\bs t=\lim_{x\rightarrow\infty}W(x)^{-1}Z(0,x)\1,\]
which similarly to the proof of Theorem~\ref{thm:limit} is a non-trivial problem.
Using recent results from~\cite{ivanovs_potential}, in particular Lemma~1, Proposition 1 and Lemma 3, we find that
this limit is given by 
\[\bs t=\mu\Delta_\bpi^{-1}\bpi_{\widetilde{\Lambda}}^T,\]
where $\bpi_{\widetilde{\Lambda}}$ is the stationary distribution associated with $\tilde \Lambda$, and the latter corresponds to the time-reversed process.
Hence 
the probability of survival according to the classical ruin concept with zero initial capital and $J(0)=i$ is given by
\begin{equation}\label{eq:classical_surv}
 \frac{\mu}{c_i}\frac{(\pi_{\widetilde{\Lambda}})_i}{\pi_i},
\end{equation}
if $X_i$ is of bounded variation, and 0 otherwise. In the case of the classical Cram\'er-Lundberg model ($n=1$) this further simplifies to the well-known expression $\mu/c$. 

The simplicity of all the terms in~\eqref{eq:classical_surv} motivates a direct probabilistic argument, which we provide in the following.
Assuming that $\mu>0$ and $X_i$ is a bounded variation process with linear drift $c_i$, 
we consider $\p_i(\tau_0^->e_q)=\p_i(\underline X(e_q)=0)$ (with an independent exponentially distributed $e_q$), which provides the required vector of survival probabilities upon taking $q\downarrow 0$. 
According to a standard time-reversal argument we write
\[\p_i(\underline X(e_q)=0|J(e_q)=j)=\tilde \p_j(X(e_q)-\overline X(e_q)=0|J(e_q)=i),\]
which yields
\begin{equation}\label{eq:timerev_ruin}
 \p_i(\tau_0^->e_q)=\sum_j\tilde \p_j(X(e_q)=\overline X(e_q),J(e_q)=i)\frac{\pi_j}{\pi_i}.
\end{equation}
Moreover
\begin{align*}
 &\tilde\p_j(X(e_q)=\overline X(e_q),J(e_q)=i)=q\,\tilde\e_j\int_0^{e_q}\ind{X(t)=\overline X(t),J(t)=i}\D t\\&=\frac{q}{c_i}\,\tilde\e_j\int_0^{\overline X(e_q)}\ind{J(\tau_x^+)=i}\D x,
\end{align*}
where the last equality follows from the structure of the sample paths (or local time at the maximum).
It is known that $\overline X(t)/t\rightarrow \mu$ as $t\rightarrow\infty$, which then shows that the above expression converges to $\frac{\mu}{c_i}(\pi_{\tilde\Lambda})_i$ as $q\downarrow 0$,
where the interchange of limit and integral can be made precise using the generalized dominated convergence theorem. Combining this with~\eqref{eq:timerev_ruin} yields~\eqref{eq:classical_surv}.

\section{A numerical example}\label{secnum}

Let us finally consider a numerical illustration of our results for a 
Markov-modulated Cram\'er-Lundberg model~\eqref{cl} with two states, exponential claim sizes with mean 1 in both states, premium densities $c_1=c_2=1$, 
claim arrival rates $\beta_1=1$, $\beta_2=0.5$, observation rates $\om_1=0.4,\om_2=0.2$,
and the Markov chain $J(t)$ having transition rates 1, 1, which results in the asymptotic drift $\mu=1/4>0$.
For this model we specify the matrix-valued functions $F(\theta)$, see~\cite[Prop.~4.2]{asm_ruin}, and $\hat F(\theta)$, cf.~\eqref{eq:hat_F}.
Using the spectral method we determine the matrices $\Lambda$ and $\hat\Lambda$, and then also the matrix $L$ according to Proposition~\ref{prop:L}: 
\begin{align*}
\Lambda=\begin{pmatrix}
-1.39& 1.39\\
1.16& -1.16              
\end{pmatrix}, \hat\Lambda=\begin{pmatrix}
-1.99& 1.20\\
1.09& -1.45             
\end{pmatrix}\,\text{and}\; L=\begin{pmatrix}
2.63& 1.47\\
1.47& 2.44            
\end{pmatrix}.
\end{align*}
We use Theorem~\ref{thm:limit} to compute the vector of survival probabilities for zero initial capital:
\begin{align*}&U=\begin{pmatrix}
1.58 & 0.58\\
0.53 & 1.54
\end{pmatrix},&  \phib(0) =U^{-1}\1=\begin{pmatrix}
0.45\\ 0.49              
\end{pmatrix}.\end{align*}

Furthermore, Corollary~\ref{cor:results} yields the vector of survival probabilities for an arbitrary initial capital $u\geq 0$ in terms of a matrix-valued function $W(x)$.
Due to the exponential jumps, the matrix $W(x)$ has an explicit form, 
which can be obtained using so-called fluid embedding to convert our model into a Markov modulated linear drift model for which $W(x)$ is known, see e.g.~\cite[Sec.~7.7]{thesis}.
Figure~\ref{fig:surv} depicts the survival probabilities as a function of the initial capital~$u$.
\begin{figure}[h!]
\caption{Survival probabilities $\phi_1(u)$ and $\phi_2(u)$.}\label{fig:surv}
\begin{center}
\includegraphics{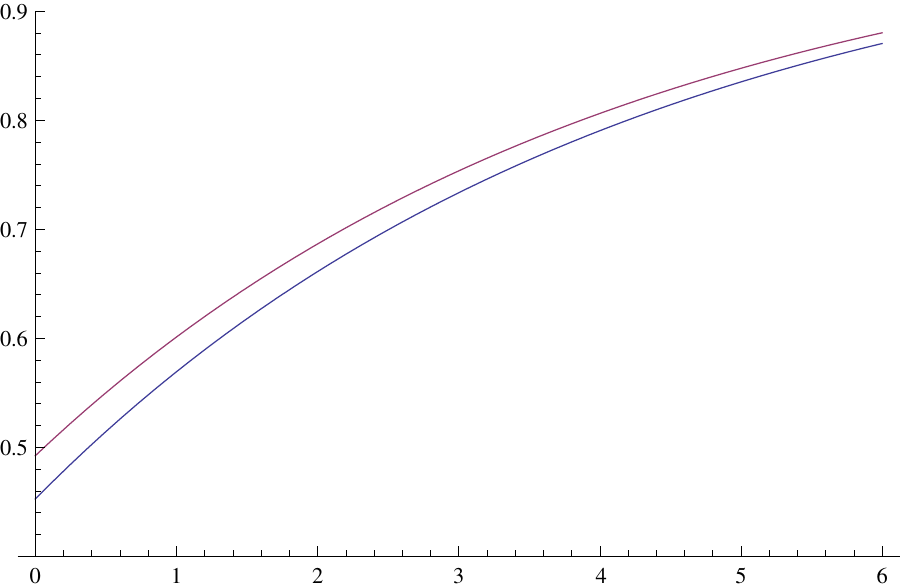}
\end{center}
\end{figure}

Figure~\ref{fig} confirms the correctness of our results. It depicts $(R(x){\bf 1})_1$ (i.e. the probability to reach level $x$ before being observed ruined when starting in state 1 with zero initial capital), 
and the dots represent Monte Carlo simulation estimates of the same quantity based on 10000 runs, the horizontal line representing $\phi_1(0)=0.45$.
One sees that for large values of $x$ the numerical determination of $R(x)$ (as well as $\bs\phi(x)$) becomes a challenge,
which underlines the importance of our limiting result, i.e.~Theorem~\ref{thm:limit}.
\newpage
\begin{figure}[h!]
\caption{Probability of reaching level $x$ before ruin for $J(0)=1,X(0)=0$.}\label{fig}
\begin{center}
\includegraphics{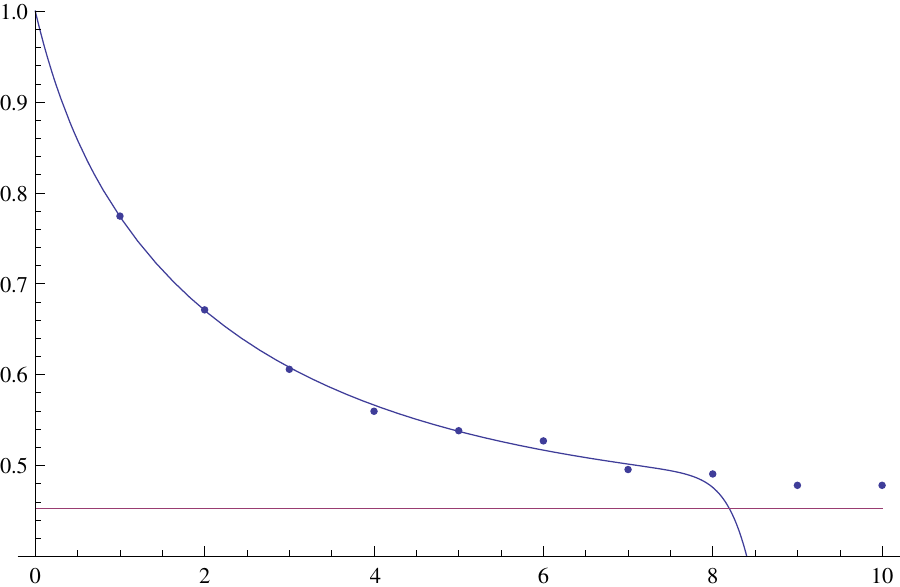}
\end{center}
\end{figure}

%
%
%
%
%
%
%
%
%
%
%
%
%
%
%
%
%
%
%
%
%
%
%
%
%
%
%


\section*{Acknowledgements}
Financial support by the Swiss National Science Foundation Project 200021-124635/1 is gratefully acknowledged.

%
%

%
%
%
%

\bibliographystyle{abbrv}
\bibliography{ruin_observer.bib}

\end{document}